\documentclass[11pt]{article}
\usepackage[utf8]{inputenc}
\usepackage[dvips]{graphicx}
\usepackage{latexsym}
\usepackage{amsmath,amssymb,amsfonts,amsthm,stmaryrd,bbm,bm}
\usepackage[dvipsnames]{xcolor}
\usepackage{enumerate} 
\usepackage{caption,tikz}
\usepackage[colorlinks=true, linkcolor=NavyBlue, urlcolor=black, citecolor=NavyBlue]{hyperref}

\newcommand \RR {\mathbb{R}}
\newcommand \ZZ {\mathbb{Z}}

\newcommand \EE {\mathbb{E}}
\newcommand \PP {\mathbb{P}}

\newcommand \QQ {\mathbb{Q}}

\newcommand{\calP}{\mathcal P}
\newcommand{\calF}{\mathcal F}
\newcommand{\bbr}{\RR}
\newcommand{\q}{\quad}
\renewcommand{\a}{\alpha}
\newcommand{\s}{\sigma}
\renewcommand{\th}{\theta}
\renewcommand{\d}{\delta}

\renewcommand{\O}{\Omega}

\renewcommand{\epsilon}{\varepsilon}
\newcommand{\bp}{{\bm{p}}}
\newcommand{\bP}{{\bm{P}}}

\theoremstyle{plain}
\newtheorem{theorem}{Theorem}[section]
\newtheorem{proposition}[theorem]{Proposition}
\newtheorem{lemma}[theorem]{Lemma}

\theoremstyle{definition}

\theoremstyle{remark}
\newtheorem{remark}[theorem]{Remark}

\let\originalleft\left
\let\originalright\right
\renewcommand{\left}{\mathopen{}\mathclose\bgroup\originalleft}
\renewcommand{\right}{\aftergroup\egroup\originalright}

\author{Serge Cohen\footnote{Institut de Math\'ematiques de Toulouse, Université de Toulouse; CNRS, F-31062 Toulouse Cedex 9, France. FirstName.Name@math.univ-toulouse.fr}
	\and James Norris\footnote{Statistical Laboratory, Centre for Mathematical Sciences, Wilberforce Road, Cambridge, CB3 0WB, UK. j.r.norris@statslab.cam.ac.uk}  
	\and Michel Pain\footnotemark[1] 
	\and Gennady Samorodnitsky\footnote{School of Operations Research and Information Engineering, Cornell University, Ithaca, NY 14853, USA. gs18@cornell.edu}
}

\title{Interlacing sequences resulting from an interval split-merge dynamics and the induced probability measures}

\begin{document}
	\maketitle
	\begin{abstract}
		We study sequences of partitions of the unit interval into subintervals, starting from the trivial partition, in which each partition is obtained from the one before by splitting its subintervals in two, according to a given rule, and then merging pairs of subintervals at the break points of the old partition. The $n$th partition then comprises $n+1$ subintervals with $n$ break points, which inherently possess an interlacing property.
		The empirical distribution of these points reveals a surprisingly rich structure, even when the splitting rule is completely deterministic. We consider both deterministic and randomized splitting rules and we study from multiple angles the limiting behavior of the empirical distribution of the break points.
	\end{abstract}
	
	
	\noindent \textit{AMS classification (2020)}: 60J05, 60F05.

\section{Introduction}

\subsection{Definition of the model}

We consider sequences of partitions $(\calP_n)_{n\ge0}$ of the interval $(0,1]$, of the form
$$
\calP_n=\{(a_{n,0},a_{n,1}],\dots,(a_{n,n},a_{n,n+1}]\}
$$
where $a_{n,0}=0$ and $a_{n,n+1}=1$ and the $n$ break points $a_{n,1},\dots,a_{n,n}$ satisfy
$$
0\le a_{n,1}\le\dots\le a_{n,n}\le1.
$$
Thus $\calP_0=(0,1]$ and $\calP_n$ partitions $(0,1]$ into $n+1$ subintervals
$$
(0,1]=(a_{n,0},a_{n,1}]\cup\dots\cup(a_{n,n},a_{n,n+1}].
$$
Suppose we are given a family of splitting proportions 
$$
\bp=(p_{n,k}:n\ge1,\,1\le k\le n)
$$ 
with $p_{n,k}\in[0,1]$ for all $k$.
For $n\ge1$, we define the break points of the partition $\calP_n$ recursively by 
\begin{equation}
\label{eq:induc-rand}
a_{n,k}=p_{n,k}a_{n-1,k-1}+(1-p_{n,k})a_{n-1,k},\q k=1,\ldots,n.
\end{equation}
Thus, we \emph{split} each interval $(a_{n-1,k-1},a_{n-1,k}]$ of the partition $\calP_{n-1}$ into two subintervals, 
with proportions $1-p_{n,k}$ on the left and $p_{n,k}$ on the right,
and then we \emph{merge} the resulting subintervals at the break points of $\calP_{n-1}$.
This split-merge dynamics is illustrated in Figure~\ref{fig:def}.
\begin{figure}[t]
	\centering
	\begin{tikzpicture}[scale=4]
	\def\r{.015}; 
	\draw (0,0) -- (1,0);
	\draw[fill,NavyBlue] (0,0) circle (\r) node[below]{0};
	\draw[fill,NavyBlue] (1,0) circle (\r) node[below]{1};
	\draw[NavyBlue] (.5,.12) node{$\calP_0$};
	\draw[->,>=latex] (1.3,0) -- (1.7,0);
	\draw (1.5,0) node[above]{\small split};
	\draw[->,>=latex] (1.7,-.2) -- (1.3,-.5);
	\draw (1.5,-.35) node[above,rotate=37]{\small merge};
	\begin{scope}[shift={(2,0)}]
	\draw (0,0) -- (1,0);
	\draw[fill,NavyBlue] (0,0) circle (\r) node[below]{0};
	\draw[fill,BrickRed] (1/3,0) circle (\r) node[below]{$\frac{1}{3}$};
	\draw[BrickRed] (.5,.1) node{\small $p_{1,1} = \frac{2}{3}$};
	\draw[fill,NavyBlue] (1,0) circle (\r) node[below]{1};
	\end{scope}
	\begin{scope}[shift={(0,-.7)}]
	\draw (0,0) -- (1,0);
	\draw[fill,NavyBlue] (0,0) circle (\r) node[below]{0};
	\draw[fill,NavyBlue] (1/3,0) circle (\r) node[below]{$\frac{1}{3}$};
	\draw[fill,NavyBlue] (1,0) circle (\r) node[below]{1};
	\draw[NavyBlue] (.5,.12) node{$\calP_1$};
	\draw[->,>=latex] (1.3,0) -- (1.7,0);
	\draw (1.5,0) node[above]{\small split};
	\draw[->,>=latex] (1.7,-.2) -- (1.3,-.5);
	\draw (1.5,-.35) node[above,rotate=37]{\small merge};
	\end{scope}
	\begin{scope}[shift={(2,-.7)}]
	\draw (0,0) -- (1,0);
	\draw[fill,NavyBlue] (0,0) circle (\r) node[below]{0};
	\draw[fill,BrickRed] (1/6,0) circle (\r) node[below]{$\frac{1}{6}$};
	\draw[BrickRed] (1/6,.1) node{\small $p_{2,1} = \frac{1}{2}$};
	\draw[fill,NavyBlue] (1/3,0) circle (\r) node[below]{$\frac{1}{3}$};
	\draw[fill,BrickRed] (5/9,0) circle (\r) node[below]{$\frac{5}{9}$};
	\draw[BrickRed] (2/3,.1) node{\small $p_{2,2} = \frac{2}{3}$};
	\draw[fill,NavyBlue] (1,0) circle (\r) node[below]{1};
	\end{scope}
	\begin{scope}[shift={(0,-1.4)}]
	\draw (0,0) -- (1,0);
	\draw[fill,NavyBlue] (0,0) circle (\r) node[below]{0};
	\draw[fill,NavyBlue] (1/6,0) circle (\r) node[below]{$\frac{1}{6}$};
	\draw[fill,NavyBlue] (5/9,0) circle (\r) node[below]{$\frac{5}{9}$};
	\draw[fill,NavyBlue] (1,0) circle (\r) node[below]{1};
	\draw[NavyBlue] (.5,.12) node{$\calP_2$};
	\draw[->,>=latex] (1.3,0) -- (1.7,0);
	\draw (1.5,0) node[above]{\small split};
	\end{scope}
	\begin{scope}[shift={(2,-1.4)}]
	\draw (.5,0) node{$\dots$};
	\end{scope}
	\end{tikzpicture}
	\caption{Illustration of the iterative definition of $(\calP_n)_{n\ge0}$.}
	\label{fig:def}
\end{figure}
By definition, the break points form an \emph{interlacing sequence}, that is they satisfy the following property, for any $n \geq 1$,
\begin{equation} \label{eq:interlacing}
a_{n,1} \leq a_{n-1,1} \leq a_{n,2} \leq a_{n-1,2} \leq \dots \leq a_{n,n-1} \leq a_{n-1,n-1} \leq a_{n,n}.
\end{equation}
The cases where $p_{n,k}=0$ or $p_{n,k}=1$ for some $n$ and $k$ are allowed: 
in such cases some intervals of the partition will be empty, but we continue to list them `with multiplicity'.
We use the same formula~\eqref{eq:induc-rand} for the dynamics of the break points in all cases.

\begin{figure}[t]
	\centering
	\begin{tikzpicture}[scale=8]
	\def\r{.01}; 
	\draw (0,0) -- (1,0) ++ (.15,0) node{$\calP_0$};
	\draw[fill,NavyBlue] (0,0) circle (\r) node[above]{$a_{0,0}=0$};
	\draw[fill,NavyBlue] (1,0) circle (\r) node[above]{$a_{0,1}=1$};
	\begin{scope}[shift={(0,-.1)}]
	\draw (0,0) -- (1,0) ++ (.15,0) node{$\calP_1$};
	\draw[fill,NavyBlue] (0,0) circle (\r);
	\draw[fill,NavyBlue] (.4,0) circle (\r);
	\draw[fill,NavyBlue] (1,0) circle (\r);
	\end{scope}
	\begin{scope}[shift={(0,-.2)}]
	\draw (0,0) -- (1,0) ++ (.15,0) node{$\calP_2$};
	\draw[fill,NavyBlue] (0,0) circle (\r);
	\draw[fill,NavyBlue] (.25,0) circle (\r);
	\draw[fill,NavyBlue] (.8,0) circle (\r);
	\draw[fill,NavyBlue] (1,0) circle (\r);
	\end{scope}
	\begin{scope}[shift={(0,-.3)}]
	\draw (0,0) -- (1,0) ++ (.15,0) node{$\calP_3$};
	\draw[fill,NavyBlue] (0,0) circle (\r);
	\draw[fill,NavyBlue] (.15,0) circle (\r);
	\draw[fill,NavyBlue] (.55,0) circle (\r);
	\draw[fill,NavyBlue] (.88,0) circle (\r);
	\draw[fill,NavyBlue] (1,0) circle (\r);
	\end{scope}
	\begin{scope}[shift={(0,-.4)}]
	\draw (0,0) -- (1,0) ++ (.15,0) node{$\calP_4$};
	\draw[fill,NavyBlue] (0,0) circle (\r);
	\draw[fill,NavyBlue] (.11,0) circle (\r);
	\draw[fill,NavyBlue] (.33,0) circle (\r);
	\draw[fill,NavyBlue] (.61,0) circle (\r);
	\draw[fill,NavyBlue] (.94,0) circle (\r);
	\draw[fill,NavyBlue] (1,0) circle (\r);
	\end{scope}
	\begin{scope}[shift={(0,-.5)}]
	\draw (0,0) -- (1,0) ++ (.15,0) node{$\calP_5$};
	\draw[fill,NavyBlue] (0,0) circle (\r) ++ (-.035,-.015) node[below]{$a_{5,0}$};
	\draw[fill,NavyBlue] (.06,0) circle (\r) ++ (0,-.015) node[below]{$a_{5,1}$};
	\draw[fill,NavyBlue] (.15,0) circle (\r) ++ (+.01,-.015) node[below]{$a_{5,2}$};
	\draw[fill,NavyBlue] (.42,0) circle (\r) ++ (0,-.015) node[below]{$a_{5,3}$};
	\draw[fill,NavyBlue] (.81,0) circle (\r) ++ (0,-.015) node[below]{$a_{5,4}$};
	\draw[fill,NavyBlue] (.96,0) circle (\r) ++ (-.01,-.015) node[below]{$a_{5,5}$};
	\draw[fill,NavyBlue] (1,0) circle (\r) ++ (+.045,-.015) node[below]{$a_{5,6}$};
	\end{scope}
	\end{tikzpicture}
	\caption{Example of the first steps of the split-merge dynamics, highlighting the interlacing property satisfied by the points $a_{n,k}$. Note that the points quickly accumulate at 0 and at 1, suggesting that the weak limit of the empirical distribution is supported by the endpoints.}
	\label{fig:interlacing}
\end{figure}

We sometimes choose the family of splitting proportions randomly.
In this case, we will use upper-case letters for both the splitting proportions $\bP=(P_{n,k}:n\ge1,\,1\le k\le n)$
and the break points $A_{n,k}$, but their relation will be the same, namely $A_{n,0}=A_{n,n+1}=0$ for all $n$ and
\begin{equation}
\label{eq:induc-rand2}
A_{n,k}=P_{n,k}A_{n-1,k-1}+(1-P_{n,k})A_{n-1,k},\q k=1,\ldots,n.
\end{equation}

In this paper, we mainly consider the following three choices for the splitting proportions.
\begin{itemize}
	\item \textbf{Deterministic stratified case:} we choose a
	deterministic sequence $(p_n)_{n\geq 1}$ in $[0,1]$ and set $p_{n,k}=p_n$ for all $k$. 
	An interesting special case is the choice $p_n\equiv p\in(0,1)$. 
	\item \textbf{Random stratified case:} we set $P_{n,k}=P_n$ for all $k$, 
	where $(P_n)_{n\ge1}$ are i.i.d.\@ random variables in $[0,1]$. 
	\item \textbf{Fully random case:} we take $(P_{n,k}:n\ge1,\,1\le k\le n)$ to be an 
	array of i.i.d.\@ random variables in $[0,1]$.   
\end{itemize}
In each of these three cases, 
the sequence $(\calP_n)_{n\ge0}$ forms a Markov chain whose state-space is the set of finite partitions 
of the unit interval into subintervals.

\subsection{Motivation and related literature}

The motivation for our work comes from an analysis of the asymptotic behavior, in a certain regime, 
of a Markov chain on $ \{-1,1\}^K $ with $K$ large, considered in~\cite{Bressaud24}. 
This latter Markov chain moves from a state $a\in\{-1,1\}^K $ to a state $ b\in\{-1,1\}^K$ according to the following rule: 
choose, uniformly at random, a subsequence $J$ of $\{1,\ldots,K\}$ 
from the set of subsequences $j=(j(1),\dots,j(\ell))$ of $\{1,\dots,K\}$ such that $a_{j(i+1)}\neq a_{j(i)}$ for all $i$.
Once a subsequence $J$ is sampled, 
the Markov chain moves to the state $b\in\{-1,1\}^K$ defined by $b_k=a_k $ if $k$ is not in the range of $J$, 
and $b_k=-a_k$ otherwise. 
Starting from  the initial state $a=(1,\dots,1)$, 
the first transition leads to a state $b$ that has at most one digit equal to $-1$, 
and this $b$ is chosen uniformly among the $K+1$ possible choices; 
furthermore, for $n\ll\log K$ the behavior of this Markov chain can be approximated by 
the fully random case of the model we consider in the present paper, 
where the proportions have the standard uniform distribution. 
For more related results concerning this Markov chain refer to~\cite{Boissard15}.

The splitting part of our dynamics appears in fragmentation processes, a well-studied class of stochastic processes (see e.g.~\cite{Bertoin01,Bertoin06,Berestycki10}),
and also bears some resemblance to the definition of multiplicative cascades (see e.g.~\cite{Kahane76,Barral2010,Heurteaux2016}).
However, the merging feature prominent in our model does not seem to have been studied yet and produces a radically different behavior. In particular, the number of intervals grows linearly in the number of steps for our model, whereas it grows exponentially for fragmentation processes or multiplicative cascades.

We have already mentioned that our dynamics produces interlacing sequences. 
We would like to point out that interlacing sequences appear naturally in other contexts as well, both deterministically and randomly. 
A common name for a family of sequences satisfying \eqref{eq:interlacing} is a \textit{Gelfand--Tsetlin pattern}, which is
in many situations a finite pattern (i.e.\@ extends over a finite range of $n$), but sometimes is infinite, like our pattern is. 
Important examples of Gelfand--Tsetlin patterns are the eigenvalues of the successive principal minors of an $N\times N$ Hermitian matrix (\cite{Hwang2004,Baryshnikov2001}) or the roots of the successive derivatives of a degree $N$ polynomial with real roots (\cite{Fisk2006}). 
These models differ from our model both in the manner interlacing structures are constructed and in the type of questions that are being asked.

\subsection{Content of the paper}

We are interested in the behavior of the empirical distribution of the break points of the partition $\calP_n$ as $n\to\infty$. 
We will set
\begin{equation} 
\label{e:emp.dist}
g_n=\frac1n\sum_{k=1}^n\delta_{a_{n,k}},\q
G_n=\frac1n\sum_{k=1}^n\delta_{A_{n,k}}
\end{equation} 
in the deterministic and random cases respectively.
Here $\delta_a$ denotes the unit mass at point $a$.  
Note that, when the splitting proportions $\bP$ are random, the empirical distribution $G_n$ is a random probability measure. 
We will study both the overall behavior of the empirical distribution and its restriction to certain parts of the unit interval. 

We begin in Section~\ref{sec:repr} by establishing a key representation of the break points in terms of the cumulative
distribution function of a certain random walk, whose steps are governed by the splitting proportions. 
We denote this walk by $(x_n)_{n\ge0}$ when we restrict to the case of deterministic splitting proportions, 
and by $(X_n)_{n\ge0}$ in the general random case.
The walk $(x_n)_{n\ge0}$ will be random in both cases, even when the splitting proportions are deterministic.
Interestingly, this representation is useful even in the deterministic case. 
In the random stratified case, $(X_n)_{n\ge0}$ becomes a random walk with a random environment in time, 
where the environment consists of the random splitting proportions $(P_n)_{n\ge1}$. 
Similarly, in the fully random case, 
$(X_n)_{n\ge0}$ becomes a random walk with a random environment in time and space. 

In Section~\ref{sec:weak-cv}, weak convergence of the empirical distribution is established, 
as a consequence of the law of large numbers satisfied by the random walk $(X_n)_{n\ge0}$.
The weak limit turns out always to be supported by the endpoints of the unit interval. 
Nonetheless, the number of points of $\calP_n$ in any given open subset of $(0,1)$ diverges,
and we proceed in Section~\ref{sec:away_endpoints} to investigate the behavior of the empirical distribution 
away from the endpoints.
This study relies on the central limit theorem satisfied by $(X_n)_{n\ge0}$ 
and we will show vague convergence, under suitable renormalization, to an infinite measure on $(0,1)$, 
whose density is the derivative of the quantile function of a standard normal distribution. 
The exact scaling is however sensitive to the case considered.
Finally, in Section~\ref{sec:near_endpoints}, we investigate more precisely the behavior of the
empirical distribution near the endpoints, 
by studying for example the number of break points of $\calP_n$ in intervals of the type $[0,x^n]$ for $0<x<1$. 
This behavior turns out to be related to the large deviation principle for the random walk $(X_n)_{n\ge0}$.

\section{Representation of the break points}
\label{sec:repr}
Set 
$$
I=\{(n,k):n\ge1,\,1\le k\le n\}
$$
and fix 
$$
\bp=(p_{n,k}:(n,k)\in I)
$$
with $p_{n,k}\in[0,1]$ for all $n$ and $k$.
Let $(U_{n,k}:(n,k)\in I)$ be a family of independent random variables, all uniformly distributed on $[0,1]$.
Set 
$$
y_{n,k}=1_{\{U_{n,k}\le p_{n,k}\}}
$$
so that $y_{n,k}$ is a Bernoulli random variable with success probability $p_{n,k}$.
Write $\O_0$ for the set of sequences $(y_{n,k}:(n,k)\in I)$ with $y_{n,k}\in\{0,1\}$ for all $n$ and $k$.
Denote by $\PP^\bp$ the law of $(y_{n,k}:(n,k)\in I)$ on $\O_0$.
Define a random process $(x_n)_{n\ge0}$ on $\O_0$ by setting $x_0=0$ and then, recursively for $n\ge1$, 
\begin{equation}
\label{XN}
x_n=x_{n-1}+y_{n,x_{n-1}+1}.
\end{equation}
Note that $x_n\in\{0,1,\dots,n\}$ for all $n$ and $(x_n)_{n\ge0}$ is a time-inhomogeneous Markov chain under $\PP^\bp$.
The following proposition provides a useful formula for the break points of the deterministic split-merge process 
with splitting proportions $\bp$.

\begin{proposition} 
	\label{prop:exact_b}
	For all $(n,k)\in I$, we have
	$$
	a_{n,k}=\PP^\bp(x_n \le k-1).
	$$
\end{proposition}
\begin{proof}
	Define for $n\ge0$ and $0\le k\le n+1$
	$$
	\a_{n,k}=\PP^\bp(x_n \le k-1).
	$$
	Then $\a_{n,0}=0$ and $\a_{n,n+1}=1$ for all $n$.
	For $(n,k)\in I$, we have
	\begin{align*}
	\a_{n,k}
	&=\PP^\bp(x_n\le k-1)\\
	&=\PP^\bp(x_{n-1}\le k-2)+\PP^\bp(x_{n-1}=k-1,y_{n,k}=0)\\
	&=p_{n,k}\a_{n-1,k-1}+(1-p_{n,k})\a_{n-1,k}.
	\end{align*}
	The same recursion~\eqref{eq:induc-rand} defines $a_{n,k}$ so $\a_{n,k}=a_{n,k}$ for all $n$ and $k$.
\end{proof}

We turn now to the more general case where the splitting proportions $\bP$ are random.
By augmenting our probability space $(\O,\calF,\PP)$ if necessary, 
we can assume that $\O$ supports a family of independent random variables $(U_{n,k}:(n,k)\in I)$ as above,
which is moreover independent of $\bP$.
Set 
$$
Y_{n,k}=1_{\{U_{n,k}\le P_{n,k}\}}
$$
and define a random process $X=(X_n)_{n\ge0}$ on $\O$ by setting $X_0=0$ and then, recursively for $n\ge1$,
$$
X_n=X_{n-1}+Y_{n,X_{n-1}+1}.
$$
Note that if $\Phi$ is a bounded measurable function on $(\ZZ^+)^{\ZZ^+}$ and $\phi:[0,1]^I\to\RR$ is defined by
$$
\phi(\bp)=\EE^\bp(\Phi(x)),
$$
where $x=(x_n)_{n\ge0}$ is the time-inhomogeneous Markov chain defined at \eqref{XN},
then $\phi$ is measurable and we have, $\PP$-almost surely,
\begin{equation} \label{eq:phiPhi}
\EE(\Phi(X)|\bP)=\phi(\bP).
\end{equation}

It will be convenient sometimes to regard \eqref{eq:induc-rand} and \eqref{e:emp.dist}
as defining measurable functions $a_{n,k}(\bp)$ and $g_n(\bp)$ of $\bp$, so we can write
$$
A_{n,k}=a_{n,k}(\bP),\q G_n=g_n(\bP).
$$
Note that, on taking $\Phi(x)=1_{\{x_n\le k-1\}}$, we have $\phi(\bp)=a_{n,k}(\bp)$ by Proposition \ref{prop:exact_b}.
Hence, the random break points $(A_{n,k}:(n,k)\in I)$ associated with the random splitting proportions $\bP$ satisfy, 
almost surely,
$$
A_{n,k}=a_{n,k}(\bP)=\PP(X_n\le k-1|\bP)
$$
We remark on two special cases where $(X_n)_{n\ge0}$ takes a simple form.
\begin{itemize}
	\item In the case where $p_{n,k}\equiv p\in[0,1]$ for all $n$ and $k$, 
	which is a special case of the deterministic stratified case, $(X_n)_{n\ge0}$ is a simple (biased) random walk. 
	In particular $X_n$ has the binomial distribution with parameters $n$ and $p$.
	\item In the fully random case, the process $(X_n)_{n\ge0}$ is a Random Walk in Dynamic Random Environment (RWDRE). 
	See for example~\cite{Rassoul17}.
	We will be interested not only in its annealed distribution, which is a simple random walk, 
	but also in its quenched distribution, conditional on the splitting proportions $\bP$.
\end{itemize}

\section{Weak convergence of the empirical distributions}
\label{sec:weak-cv}
Our first use of the representation of the partitions $\calP_n$ developed in the previous section shows that, 
perhaps somewhat unexpectedly, 
under quite general conditions, 
the sequence of empirical distributions of break points $(g_n)_{n\ge0}$ (or $(G_n)_{n\ge0}$ in the random case) 
has a weak limit supported on $\{0,1\}$.
This is the content of the main result of this section, Theorem \ref{thm:cv_p} below.

We start with a key lemma stating that this weak convergence is a consequence of 
a law of large numbers for the auxiliary walk $(x_n)_{n\ge0}$.

\begin{lemma}
	\label{lem:LLN}
	Fix a family of splitting proportions $\bp=(p_{n,k}:(n,k)\in I)$ 
	and consider the associated random process $(x_n)_{n\ge0}$ defined on $\O_0$ by \eqref{XN}. 
	Assume that, for some $\bar p\in[0,1]$, under $\PP^\bp$, as $n\to\infty$,
	\begin{equation} 
	\label{e:condA}
	x_n/n\to\bar p\quad\text{in probability}.
	\end{equation} 
	Then, as $n\to\infty$, 
	$$
	g_n\Rightarrow\bar p\delta_0+(1-\bar p)\delta_1\q\text{weakly on $[0,1]$}.
	$$
\end{lemma}
\begin{proof}
	It follows from the assumption that, for any fixed $\a\in(0,1)$, by Proposition \ref{prop:exact_b}, 
	\[
	a_{n,\lfloor\a n\rfloor}=\PP^\bp(x_n\le\lfloor\a n\rfloor-1)\to
	\begin{cases}
	0&\text{if }\a<\bar p,\\
	1&\text{if }\a>\bar p,\\
	\end{cases}
	\] 
	and so $g_n([0,x])\to\bar p$ for all $x\in(0,1)$, which implies the claimed weak convergence.
\end{proof}

\begin{theorem}	\label{thm:cv_p}
	\textup{(a)} {\bf Deterministic stratified case}. 
	Assume that the averages $\frac1n\sum_{k=1}^np_k$ converge to some limit $\bar p$ as $n\to\infty$.
	Then, as $n\to\infty$, 
	$$
	g_n\Rightarrow\bar p\delta_0+(1-\bar p)\delta_1\q\text{weakly on $[0,1]$}.
	$$
	\textup{(b)} {\bf Random stratified and fully random cases}. 
	Define $\bar p=\EE(P_1)$ or $\bar p=\EE(P_{1,1})$ according to the case in hand.
	Then, almost surely, as $n\to\infty$, 
	$$
	G_n\Rightarrow\bar p\delta_0+(1-\bar p)\delta_1\q\text{weakly on $[0,1]$}.
	$$
\end{theorem}
\begin{proof}
	For part (a), under $\PP^\bp$, the process $(x_n)_{n\ge0}$ is a sum of independent Bernoulli random variables, 
	with success probabilities $p_1,p_2,\dots$.
	Then, by Hoeffding's inequality, for all $\varepsilon>0$, 
	\[
	\PP^\bp\left(\left|\frac{x_n}n-\frac1n\sum_{k=1}^np_k\right|\ge\varepsilon\right)
	\le2\exp\left\{-2n\varepsilon^2\right\}.
	\]
	Hence $x_n/n\to\bar p$ in probability and Lemma \ref{lem:LLN} applies. 
	
	For part (b), note that, in both cases, 
	the process $(X_n)_{n\ge0}$ is a random walk with mean step size $\bar p$,
	so $X_n/n\to\bar p$ almost surely by the strong law of large numbers.
	Set
	$$
	\Phi(x)=1_{\{x_n/n\to\bar p\text{ as }n\to\infty\}}
	$$
	then
	$$
	\phi(\bp)=\EE^\bp(\Phi(x))=\PP^\bp(x_n/n\to\bar p\text{ as }n\to\infty).
	$$
	By \eqref{eq:phiPhi}, almost surely,
	$$
	\phi(\bP)=\PP(X_n/n\to\bar p\text{ as }n\to\infty|\bP).
	$$
	Hence $\phi(\bP)=1$ almost surely.
	Consider the event $\O^*=\{\bP\in B\}$, where
	$$
	B=\{\bp:x_n/n\to\bar p\text{ in probability as $n\to\infty$ under $\PP^\bp$}\}.
	$$
	By Lemma \ref{lem:LLN}, for $\bp\in B$, we have $g_n(\bp)\Rightarrow\bar p\d_0+(1-\bar p)\d_1$ weakly as $n\to\infty$.
	Then, since convergence almost surely implies convergence in probability, 
	$$
	\PP(\O^*)\ge\PP(\phi(\bP)=1)=1
	$$
	and on $\O^*$ we have 
	$$
	G_n=g_n(\bP)\Rightarrow \bar p\d_0+(1-\bar p)\d_1\q,
	$$
	weakly as $n\to\infty$.
\end{proof}

\section{Taking a closer look}
\label{sec:renorm}
The main result of the previous section, Theorem \ref{thm:cv_p}, shows that, under general conditions, 
the empirical probability distribution of break points $G_n$, defined at \eqref{e:emp.dist}, converges weakly on $[0,1]$ 
(on an event of probability 1 in the random cases) to a limit supported by the endpoints of the interval. 
That is, most of the break points $A_{n,k}$ `move' either to 0 or to 1 as the split-merge dynamics proceeds. 
Still, it is clear that this result does not provide a completely satisfactory picture of the distribution of the break points, 
even for large $n$. 
For example, as $n\to\infty$, there will be more and more break points away from the endpoints of the interval. 
How are these points distributed? 
Furthermore, even though most of the mass of $G_n$ moves towards the endpoints, how fast does this movement happen?  
We address these questions in this section. 

\subsection{Away from the endpoints}
\label{sec:away_endpoints}
The results in this subsection will be expressed in terms of the quantile function 
$$
Q=\Phi^{-1}:(0,1)\to\RR
$$ 
associated with the standard normal cumulative distribution function $\Phi$.


\begin{lemma} \label{lem:CLT}
	Fix a family of splitting proportions $\bp=(p_{n,k}:(n,k)\in I)$ 
	and write $(g_n)_{n\ge0}$ for the associated sequence of empirical distributions of the break points.
	Consider the probability measure $\PP^\bp$ on $\O_0$ defined in Section~\ref{sec:repr} and the random process $(x_n)_{n\ge0}$ on $\O_0$ defined at \eqref{XN}. 
	Assume that there exist sequences $(m_n)_{n\ge0}$ in $\RR$ and $(\sigma_n)_{n\ge0}$ in $(0,\infty)$, 
	such that, as $n\to\infty$, we have $\sigma_n\to\infty$ and, for all $t\in\RR$,
	\begin{equation} 
	\label{eq:hyp_CLT}
	\PP^\bp\left(\frac{x_n-m_n}{\sigma_n}\le t\right)\to\Phi(t).
	\end{equation}
	Then, for all $x,y\in(0,1)$ with $x\le y$, as $n\to\infty$,
	$$
	\frac{ng_n([x,y])}{\sigma_n}\to Q(y)-Q(x).
	$$
	That is, the sequence $(ng_n/\sigma_n)$ of Radon measures on $(0,1)$
	converges vaguely on that space to the Radon measure with density $Q'$ with respect to the Lebesgue measure. 
\end{lemma}
\begin{proof}
	For $x\in(0,1)$, define
	\[
	k_n(x)=\min\{k\in\{0,1,\dots,n+1\}:a_{n,k}\ge x\}.
	\]
	Then
	\begin{equation} 
	\label{eq:using_k_n(x)}
	\frac{ng_n([0,x))}{\s_n}=\frac{k_n(x)}{\sigma_n}.
	\end{equation}
	We now focus on estimating the asymptotic behavior of $k_n(x)$.
	
	By Proposition \ref{prop:exact_b}, we have
	\[
	a_{n,k_n(x)} 
	=\PP^\bp\left(x_n\le k_n(x)-1\right)
	=\PP^\bp\left(\frac{x_n-m_n}{\sigma_n}\le \frac{k_n(x)-1-m_n}{\sigma_n}\right).
	\]
	Moreover, the convergence in \eqref{eq:hyp_CLT} actually holds uniformly in $t\in\RR$ by Pólya's theorem 
	(pointwise convergence of a sequence of bounded non-decreasing functions on $\RR$ towards a continuous function implies uniform convergence).
	So we get
	\begin{equation} 
	\label{eq:1st_step_CLT}
	\left\lvert a_{n,k_n(x)}
	-\Phi\left(\frac{k_n(x)-1-m_n}{\sigma_n}\right)\right\rvert\to0.
	\end{equation}
	On the other hand, by definition of $k_n(x)$, we have $x\le a_{n,k_n(x)}\le x+\ell_n$, 
	where  $\ell_n=\max_{i=0}^n(a_{n,i+1}-a_{n,i})$ is the length of the longest interval in the partition $\calP_n$. 
	Moreover, by Proposition \ref{prop:exact_b} again
	\begin{align}
	a_{n,i+1}-a_{n,i}
	&\le\Phi\left(\frac{i-m_n}{\sigma_n}\right)-\Phi\left(\frac{i-1-m_n}{\sigma_n}\right) \nonumber \\
	&\qquad {} + 2 \sup_{t\in\RR} 
	\left\lvert \PP^\bp\left( \frac{X_n-m_n}{\sigma_n} \leq t \right) - \Phi(t) \right\rvert. \label{eq:ref_control_ell_n}
	\end{align}
	Using \eqref{eq:hyp_CLT} and that $\Phi$ is Lipschitz on $\RR$ together with $\sigma_n \to \infty$, 
	we deduce that $a_{n,i+1} - a_{n,i} \to 0$ as $n \to \infty$ uniformly in $i$, which implies that $\ell_n \to 0$.
	In particular, we get that $a_{n,k_n(x)} \to x$ as $n\to\infty$ and, coming back to \eqref{eq:1st_step_CLT}, this yields
	\begin{equation} 
	\label{eq:2nd_step_CLT}
	\left\lvert x
	- \Phi \left( \frac{k_n(x)-1-m_n}{\sigma_n} \right) \right\rvert 
	\xrightarrow[n\to\infty]{} 0.
	\end{equation}
	Using continuity of $Q$ at $x$ and the fact that $Q=\Phi^{-1}$, we get
	\begin{equation} 
	\label{eq:3rd_step_CLT}
	\left\lvert Q(x)
	-\frac{k_n(x)-1-m_n}{\sigma_n}\right\rvert\to0
	\end{equation}
	and therefore $k_n(x)=m_n+Q(x)\sigma_n+o(\sigma_n)$.
	Coming back to \eqref{eq:using_k_n(x)}, we get, in the limit $n\to\infty$,
	\begin{equation} 
	\label{eq:expansion_G_n([0,x))}
	\frac{ng_n([0,x))}{\s_n}=\frac{m_n}{\sigma_n}+Q(x)+o(1)
	\end{equation}
	It follows that, for any $x,y\in(0,1)$ with $x\le y$, we have 
	$$
	\frac{ng_n([x,y])}{\s_n}\to Q(y)-Q(x).
	$$
\end{proof}

\begin{remark}
	Consider the following strengthening of \eqref{eq:hyp_CLT}: there is a constant $C<\infty$ such that, for all $n$, we have
	\begin{equation} 
	\label{eq:hyp_CLT_2}
	\sup_{t\in\RR}\left\lvert\PP^\bp\left(\frac{x_n-m_n}{\sigma_n}\le t\right)-\Phi(t)\right\rvert\le\frac C{\sigma_n}
	\end{equation}
	Such an assumption would correspond to a Berry--Esseen bound for the walk $(x_n)_{n\ge0}$ 
	and is known%
	\footnote{For the deterministic stratified case, with $m_n = \frac{1}{n} \sum_{k=1}^n p_k$ and $\sigma_n^2 = \sum_{k=1}^n p_k(1-p_k)$, by \cite[Theorem XVI.5.2]{Feller1971}, the left-hand side of \eqref{eq:hyp_CLT_2} is at most $6r_n/\sigma_n^3$ with $r_n = \sum_{k=1}^n p_k(1-p_k)(p_k^2+(1-p_k)^2)$. Noting that $r_n \leq 2 \sigma_n^2$ yields \eqref{eq:hyp_CLT_2} with $C=12$. As a consequence, the same bound holds a.s.\@ in the random stratified case.}
	under the assumptions of Theorems~\ref{the:renorm-pn} and~\ref{th:strat-rand-renorm-bulk}.
	Then, the same proof gives a rate of convergence: for any $\delta\in(0,1/2)$, uniformly in $\delta\le x<y\le 1-\delta$,
	\begin{equation*}	
	\frac{ng_n([x,y])}{\s_n}
	=Q(y)-Q(x)+O\left(\frac1{\sigma_n}\right).
	\end{equation*}
	The restriction to the interval $[\delta,1-\delta]$ is required to have $Q$ Lipschitz 
	in the step from \eqref{eq:2nd_step_CLT} to \eqref{eq:3rd_step_CLT}.
	Moreover, \eqref{eq:ref_control_ell_n} shows that under this stronger assumption $\ell_n=O(1/\sigma_n)$.
\end{remark}

We now state and prove convergence results for our three main cases of our interval split-merge dynamics.
The first result addresses the deterministic stratified case. 
Note that the assumptions of this result may hold even when the assumption in part (a) of Theorem \ref{thm:cv_p} does not hold. 
Recall that we write $Q$ for the quantile function of the standard normal distribution.

\begin{theorem}
	\label{the:renorm-pn}
	{\bf Deterministic stratified case:} 
	Set
	$$
	s_n=\sum_{k=1}^n p_k (1-p_k)
	$$  
	and assume that $s_n\to\infty$ as $n\to\infty$.
	Then the sequence $(ng_n/\sqrt{s_n})$ of Radon measures on $(0,1)$
	converges vaguely on that space to the Radon measure with density $Q'$ with respect to the Lebesgue measure. 
\end{theorem}
\begin{proof}
	In this case, under $\PP^\bp$, 
	$x_n$ is a sum of independent Bernoulli random variables with success probabilities $p_1,\dots,p_n$, 
	so \eqref{eq:hyp_CLT} holds by Lindeberg's central limit theorem, with 
	$$
	m_n=\sum_{k=1}^np_k,\q
	\sigma_n=\sqrt{s_n}.
	$$
	Hence the result follows from Lemma \ref{lem:CLT}.
\end{proof}

We obtain as a corollary the following result for the random stratified case.
This can alternatively be shown directly from Lemma \ref{lem:CLT}. 

\begin{theorem}
	\label{th:strat-rand-renorm-bulk}
	{\bf Random stratified case:} 
	Set
	$$
	s=\EE(P_1(1-P_1))
	$$
	and assume that $\PP(P_1\in(0,1))>0$.
	Then, almost surely, the sequence $(\sqrt{n}G_n/\sqrt{s})$ of Radon measures on $(0,1)$
	converges vaguely on that space to the Radon measure with density $Q'$ with respect to the Lebesgue measure. 
\end{theorem}
\begin{proof} 
	Condition on the stratified family of splitting proportions $(P_n)_{n\ge1}$ and write
	\begin{align*}
	\frac{\sqrt nG_n}{\sqrt s}
	=\frac{nG_n}{\sqrt{\sum_{k=1}^n P_k(1-P_k)}}\times\sqrt{\frac{\sum_{k=1}^n P_k(1-P_k)}{ns}}.
	\end{align*}
	Theorem~\ref{the:renorm-pn} applies to the first factor on the right, 
	while the second factor tends to $1$ almost surely by the strong law of large numbers.
\end{proof}

Our final result in this subsection describes what happens in the case of fully random case.

\begin{theorem}
	\label{the:renorm-random-full}
	{\bf Fully random case:} 
	Set 
	$$
	\bar p=\EE(P_{1,1}),\q s=\bar p(1-\bar p)
	$$
	and assume that $\PP(P_{1,1}\in(0,1))>0$.
	Then, almost surely, the sequence $(\sqrt{n}G_n/\sqrt{s})$ of Radon measures on $(0,1)$
	converges vaguely on that space to the Radon measure with density $Q'$ with respect to the Lebesgue measure. 
\end{theorem}
\begin{proof}
	Set
	$$
	m_n=np,\q\sigma_n=\sqrt{ns},
	$$
	and consider the event $\O^*=\{\bP\in B\}$, where
	$$
	B=\{\bp:\PP^\bp((x_n-m_n)/\s_n\le t)\to\Phi(t)\text{ for all $t\in\RR$}\}.
	$$
	This event can be rewritten as follows
	$$
	\O^*=\{\PP((X_n-m_n)/\s_n\le t|\bP)\to\Phi(t)\text{ for all }t\in\QQ\}.
	$$
	Therefore, by the almost sure central limit theorem for the RWDRE, which has been proved in~\cite{Rassoul05}, we have $\PP(\O^*)=1$. 
	In particular, the assumption $\PP(P_{1,1}\in(0,1))>0$ implies the ellipticity assumption, Hypothesis (ME), in that paper.
	Then, by Lemma \ref{lem:CLT}, on $\O^*$, for all $x,y\in(0,1)$ with $x\le y$,
	$$
	\frac{\sqrt nG_n([x,y])}{\sqrt s}=\frac{ng_n(\bP)([x,y])}{\s_n}\to Q(y)-Q(x),
	$$
	which concludes the proof.
\end{proof}


\begin{remark}
	The limits in Theorem \ref{thm:cv_p} depend only on the `average' splitting proportions. 
	Comparing with Theorems \ref{the:renorm-pn}, \ref{th:strat-rand-renorm-bulk} and \ref{the:renorm-random-full}, 
	on the other hand, shows that the speed of convergence to the limits in Theorem~\ref{thm:cv_p} depends on other factors. 
	For example:
	\begin{itemize}
		\item choosing $p_n=1/2$ for all $n$ in the deterministic stratified case, 
		we get $s_n=n/4$, so the scaling factor in Theorem~\ref{the:renorm-pn} becomes $2\sqrt n$;
		\item choosing $P_1$ to be uniformly distributed on $[0,1]$ in the random stratified case,
		we get $s=1/6$, so the scaling factor in Theorem~\ref{th:strat-rand-renorm-bulk} becomes $\sqrt{6n}$;
		\item choosing $P_{1,1}$ to be uniformly distributed on $[0,1]$ in the fully random case, 
		we get $s=1/4$, so the scaling factor in Theorem~\ref{the:renorm-random-full} becomes $2\sqrt n$.
	\end{itemize}
	This indicates that there is a smaller residual mass in $G_n$ away from the endpoints of the unit interval 
	in the random stratified case than in the two other cases, 
	hence a higher speed of `escape' of the mass to the endpoints, though not by an order of magnitude.
	Furthermore, it can look surprising that the fully random case has exactly the same multiplicative factor as the case $p_n=1/2$: 
	this can be explained by the fact that, for later split-merge steps, a large number of random variables are involved, 
	resulting in an averaging effect.
\end{remark}

\subsection{Near the endpoints}
\label{sec:near_endpoints}
We now look more closely at how the mass of the empirical distribution of break points 
moves towards the endpoints of the interval $[0,1]$. 
We make statements for the behavior close to 0.
Analogous statements near 1 can be obtained by flipping the interval.

Consider the transformed empirical distributions $\tilde g_n$ defined by
$$
\tilde g_n=\frac1n\sum_{k=1}^n\d_{\tilde a_{n,k}},\q \tilde a_{n,k}=a_{n,k}^{1/n}.
$$
As usual, we will change to upper-case and write $\tilde G_n$ when the splitting proportions are random.
Note that $\tilde g_n$ is simply the pushforward of $g_n$ by $x \mapsto x^{1/n}$ and the following relation holds, for $x,y\in[0,1]$ with $x\le y$,
$$
\tilde g_n([x,y])=g_n([x^n,y^n]).
$$
We start with a lemma relating the limiting behavior of $\tilde g_n$ to large deviation estimates on the lower tail
for the associated random walk.

\begin{lemma} 
	\label{lem:LDP}
	Fix a family of splitting proportions $\bp=(p_{n,k}:(n,k)\in I)$. 
	Consider the probability measure $\PP^\bp$ on $\O_0$ defined in Section \ref{sec:repr} and the
	random process $(x_n)_{n\ge0}$ on $\O_0$ defined at \eqref{XN}. 
	Assume that there exists $\bar p\in(0,1]$ and a function $I:(0,\bar p]\to[0,\infty)$, 
	continuous and decreasing with $I(\bar p)=0$, and such that, for all $\a\in(0,\bar p)$, 
	\begin{equation}  
	\label{eq:hyp_LDP}
	\frac1n\log\PP^\bp\left(x_n\le\a n\right)\to-I(\alpha),\q\text{ as $n\to\infty$}. 
	\end{equation} 
	Define $I(0)\in(0,\infty]$ and $x_*\in[0,1)$ by
	$$
	I(0)=\lim_{\alpha \to 0^+}I(\alpha),\q x_*=e^{-I(0)}.
	$$
	For $x\in[x_*,1]$, let $\alpha_I(x)$ be the unique solution $\alpha\in[0,\bar p]$ of the equation $I(\alpha)=\log(1/x)$.
	Write $\tilde g$ for the probability measure on $[0,1]$ with distribution given by
	\begin{equation} 
	\label{eq:limit_x^n}
	\tilde g([0,x])=
	\begin{cases} 
	0,&\text{if }x\in[0,x_*], \\
	\alpha_I(x),&\text{if }x\in(x_*,1).
	\end{cases}
	\end{equation}
	Then, as $n\to\infty$, we have
	$$
	\tilde g_n\Rightarrow\tilde g\q\text{weakly on $[0,1]$}.
	$$
\end{lemma}

\begin{remark} 
	\label{rem:alpha_I}
	Under the assumptions of the lemma, the map $I:[0,\bar p]\to[0,\log(1/x_*)]$ is one-to-one, 
	which implies that $\alpha_I(x)$ is well-defined for $x\in[x_*,1]$. 
	Moreover, $\alpha_I$ is continuous and increasing on $[x_*,1]$ with $\alpha_I(x_*)=0$ and $\alpha_I(1)=\bar p$.
	Hence $x\mapsto\tilde g([0,x])$ is continuous on $[0,1)$, so $g$ is atomless on $[0,1)$, and
	$$
	\tilde g(\{1\})=1-\bar p.
	$$
\end{remark}

\begin{proof}
	Since $\tilde g_n$ and $\tilde g$ are supported on $[0,1]$, and $\a_I$ is continuous on $[x_*,1]$ with $\a_I(x_*)=0$,
	it will suffice to show that $\tilde g_n([0,x])\to\a_I(x)$ for all $x\in(x_*,1)$.
	Fix $x\in(x_*,1)$ and note that
	\begin{equation*} 
	\tilde g_n([0,x])
	=\frac1n\max\left\{k\ge0:a_{n,k}^{1/n}\le x\right\}. 
	\end{equation*}
	We use Proposition \ref{prop:exact_b} and \eqref{eq:hyp_LDP}, together with continuity of $I$, to see that, 
	for all $\a\in(0,\bar p)$, as $n\to\infty$,
	\begin{equation*} 
	\log a_{n,\lfloor\alpha n\rfloor}^{1/n} 
	=\frac1n\log a_{n,\lfloor\alpha n\rfloor} 
	=\frac1n\log\PP^\bp(x_n\le\lfloor\alpha n\rfloor-1)
	\to-I(\alpha)
	\end{equation*} 
	so
	$$
	a_{n,\lfloor\alpha n\rfloor}^{1/n}\to e^{-I(\a)}.
	$$
	This implies the desired limit for $\tilde g_n([0,x])$.
\end{proof}

We now apply this lemma to the case of deterministic stratified case. 

\begin{theorem}
	\label{th: renorm-bound}
	{\bf Deterministic stratified case:} 
	Assume that, for some probability measure $H$ on $[0,1]$ with $H(\{0,1\})<1$, 
	the following weak limit holds on $[0,1]$ as $n\to\infty$,
	\begin{equation} 
	\label{e:ass.H}
	\frac1n\sum_{k=1}^n\delta_{p_k}\Rightarrow H.
	\end{equation}
	Set
	\[
	\bar p=\int_0^1t\,H(dt).
	\]
	Then $\bar p\in(0,1)$ and, for all $\a\in(0,\bar p)$, there is a unique $\th=\th(\a)\in\RR$ such that
	\begin{equation} \label{e:theta}
	\int_0^1\frac{te^\th}{1-t+te^\th}\,H(dt)=\a.
	\end{equation} 
	Define, for $\a\in(0,\bar p)$,
	\[
	I(\alpha)=\alpha\theta(\alpha)-\int_0^1\log\bigl(1-t+te^{\theta(\alpha)}\bigr)\,H(dt).
	\]
	Then $I$ is continuous and decreasing on $(0,\bar p)$ and $I(\a)\to0$ as $\a\to\bar p$.
	Moreover, in the limit $n\to\infty$, we have
	$$
	\tilde g_n\Rightarrow\tilde g\q\text{ weakly on $[0,1]$}
	$$
	where $\tilde g$ is defined as in Lemma \ref{lem:LDP}.
\end{theorem}

\begin{remark} 
	\label{rk:det.endpoints}
	The hypothesis \eqref{e:ass.H} implies in particular that $\frac1n\sum_{k=1}^np_k$ converges to $\bar p$ as $n\to\infty$,
	and hence, by Theorem \ref{thm:cv_p}, that
	$$
	g_n\Rightarrow\bar p\d_0+(1-\bar p)\d_1\q\text{weakly on $[0,1]$}.
	$$
	Moreover \eqref{e:ass.H} is also symmetric under flipping the interval $[0,1]$, so \eqref{e:ass.H} implies
	an analogous limiting statement for $g_n([1-x^n,1])$.
\end{remark}

\begin{remark} 
	It is straightforward to check from \eqref{e:theta} that as $\a\to0$ we have $\th(\a)\to-\infty$ and $\a\th(\a)\to0$, so 
	$$
	\lim_{\a\to0} I(\a) = -\int_0^1\log(1-t)\,H(dt). 
	$$
	Hence, in the notation of Lemma~\ref{lem:LDP}, we have
	\begin{equation} 
	\label{e:xstar}
	x_*=\exp\left\{\int_0^1\log(1-t)\,H(dt)\right\}\in[0,1-\bar p].
	\end{equation}
\end{remark}

\begin{proof}
	We check the assumptions of Lemma \ref{lem:LDP}. 
	Note that, for $\theta\in\bbr$, 
	\begin{align*}
	\Lambda_n(\theta):=\log\EE^\bp(e^{\theta x_n/n})
	=\sum_{i=1}^n\log(1-p_i+p_ie^{\theta/n})
	\end{align*} 
	so it follows by the assumption \eqref{e:ass.H} that 
	$$
	\frac1n\Lambda_n(n\theta)
	\xrightarrow[n\to\infty]{}\int_0^1\log\bigl(1-x+xe^\theta\bigr)\,H(dx)=:\Lambda(\theta).
	$$
	Since $\Lambda$ is finite and differentiable on $\bbr$, 
	it follows from the G{\"a}rtner-Ellis theorem (see e.g.\@ Theorem 2.3.6 in \cite{Dembo10}) 
	that, under $\PP^\bp$, the sequence $(x_n/n)_{n\ge1}$ satisfies a large deviation principle with speed $n$ and rate function
	$$
	\Lambda^*(\alpha)=\sup_{\theta\in\bbr}\left(\theta\alpha-\int_0^1\log\bigl(1-x+xe^\theta\bigr)\,H(dx)\right)
	$$
	which is the Fenchel--Legendre transform of $\Lambda$.
	In particular, 
	the function $\Lambda^*:\RR\to[0,\infty]$ is convex and has a unique minimum value equal to $0$ at $\Lambda'(0)=\bar p$.
	If $\alpha\in[0,\bar p]$, then the supremum in the definition of $\Lambda^*$ is attained at $\theta=\theta(\alpha)$ 
	and therefore $I(\alpha)=\Lambda^*(\alpha)$.
	It follows that $I$ satisfies the assumptions of Lemma~\ref{lem:LDP}:  
	the function $I:(0,\bar p]\to[0,\infty)$ is continuous, decreasing, 
	and satisfies $I(\bar p)=0$ and \eqref{eq:hyp_LDP} as a consequence of the large deviation principle mentioned above.
	Therefore, we can apply Lemma~\ref{lem:LDP} which yields the result.
\end{proof}

Theorem \ref{th: renorm-bound} immediately allows us to deduce the following result.

\begin{theorem}
	\label{th:renorm-rand-strat}
	{\bf Random stratified case:} 
	Consider the random stratified case with splitting proportions $(P_n)_{n\ge1}$.
	Denote by $H$ the distribution of $P_1$ and assume that $H(\{0,1\})<1$.
	Then $\tilde G_n\Rightarrow\tilde g$ weakly on $[0,1]$, almost surely as $n\to\infty$, 
	where $\tilde g$ is defined as in Theorem~\ref{th: renorm-bound}.
\end{theorem} 
\begin{proof}
	It is enough to note that, by the Glivenko--Cantelli theorem, as $n\to\infty$, almost surely,
	$$
	\frac1n\sum_{k=1}^n\d_{P_k}\Rightarrow H\q\text{ weakly on $[0,1]$}.
	$$
	Then Theorem \ref{th: renorm-bound} applies on the same set of full probability.
\end{proof} 

We finish this section with a discussion of the fully random case. 
The overall picture of the movement of the mass of $G_n$ towards the endpoints of the interval $[0,1]$ remains the same
as in the deterministic stratified case and the random stratified case. 
However, the assumptions are stronger and less information about the limiting distribution is available. 

\begin{theorem}
	\label{th: renorm-rand-full}
	{\bf Fully random case:} 
	Consider the fully random case with splitting proportions $(P_{n,k}:(n,k)\in I)$.
	Assume that for some $p>2$ we have
	\begin{equation}
	\label{eq:ellip}
	\EE\left(\lvert\log(P_{1,1})\rvert^p\right)<\infty
	\quad\text{and}
	\quad
	\EE\left(\lvert\log(1-P_{1,1})\rvert^p\right)<\infty. 
	\end{equation} 
	Then, for some probability distribution $\tilde g$ on $[0,1]$, 
	we have $\tilde G_n\Rightarrow\tilde g$ weakly on $[0,1]$, almost surely as $n\to\infty$.
	Moreover $\tilde g$ is supported on $[x_*,1]$ for some $x_*\in(0,1)$, $\tilde g$ has no atoms in $[0,1)$ and $\tilde{g} (\{1\}) = 1 - \EE(P_{1,1})$.
\end{theorem} 
\begin{proof}
	Set $\bar p=\EE(P_{1,1})$.
	It follows from e.g.\@ Corollary 2.3 and Proposition 2.4 in \cite{Rassoul17} 
	that there is a function $I_q:[0,\bar p]\to[0,\infty)$, 
	continuous and decreasing with $I_q(\bar p)=0$, and such that, almost surely, for all $\a\in(0,\bar p)$ as $n\to\infty$,
	\begin{equation}  
	\label{eq:hyp_LDPQ}
	\frac1n\log\PP\left(X_n\le\a n|\bP\right)\to-I_q(\alpha).
	\end{equation} 
	We use here the assumption \eqref{eq:ellip}.
	The subscript $q$ here is for `quenched'.
	Hence, on the same event of full probability, 
	we can apply Lemma \ref{lem:LDP} to deduce the claimed weak convergence of $\tilde G_n$.
	The fact that $x_*>0$ in this case follows from the finiteness of $I_q(0)$ guaranteed in \cite{Rassoul17}.
	Stated properties of $\tilde{g}$ follow from Remark \ref{rem:alpha_I}.
\end{proof}

\begin{remark}
	The limit distribution $\tilde g$ in Theorem \ref{th: renorm-rand-full} is given in Lemma~\ref{lem:LDP} in terms of $I_q$.
	Since a full description of the rate function $I_q$ is not available, neither is a full description of $\tilde g$.
	Nonetheless, some information is available in \cite{Rassoul17}.
	Under $\PP$, that is in the annealed case, $(X_n)_{n\ge0}$ is a random walk with i.i.d.\@ Bernoulli($\bar p$) jumps.
	So $(X_n/n)_{n\ge1}$ satisfies a large deviation principle with rate function $I_a$, and for $\alpha\in[0,1]$,
	\[
	I_a(\alpha)=\alpha\log\frac{\alpha}{\bar p}+(1-\alpha)\log\frac{1-\alpha}{1-\bar p}.
	\]
	Proposition 2.4 in~\cite{Rassoul17} says that $I_a(\alpha)\le I_q(\alpha)$ for $0\le\alpha\le1$, 
	and that the two functions share the same unique global minimum, at the point $\bar p$, where they both vanish. 
	Let
	$$
	x_{*,a}=e^{-I_a(0)}=1-\bar p\,\in(0,1).
	$$
	Define a probability law supported by $[x_{*,a},1]$ by the cumulative distribution function 
	$$
	F_a(x)=\frac{\alpha_a(x)}{\bar p}, \ x\in[x_{*,a} ,1]
	$$
	where $\alpha_a(x)$ is the unique solution in $[0,\bar p]$ of the equation 
	$$
	I_a(\alpha)=\log(1/x). 
	$$
	Then the law $F_a$ stochastically dominates the limiting distribution in Theorem \ref{th: renorm-rand-full}. 
	Furthermore, unless $P_{1,1}$ is non-random, according to the same reference we have
	$I_a(0)<I_q(0)$, and so we have a strict ordering of the left endpoints of the support: $x_*<1-\bar p$. 
\end{remark}

\section*{Acknowledgments} The authors thank two anonymous referees for useful comments and for pointing out the link with Gelfand--Tsetlin patterns. Professor Cohen would like to thank Labex CIMI for its support.

\bibliographystyle{plain}
\bibliography{biblio}

\end{document}